\documentclass[11pt, reqno]{amsart}
\usepackage{amsthm, amscd, amsfonts, amssymb, graphicx,tikz, color, xcolor, fancybox}
\usepackage[all]{xy}
\usepackage{hyperref}
\usepackage{enumerate}
\usepackage{comment}
\newcommand{\kk}{{\Bbbk}}
\newcommand{\GL}{\mbox{\rm GL}}
\newcommand{\SL}{\mbox{\rm SL}}

\newcommand{\C}{{\mathbb C}}
\newcommand{\R}{{\mathbb R}}
\newcommand{\FF}{{\mathbb F}}
\newcommand{\Q}{{\mathbb Q}}
\newcommand{\Z}{{\mathbb Z}}

\newcommand{\HH}{{\mathbb H}}

\newcommand{\f}{\mathfrak}
\newcommand{\car}{\mbox{\rm char\,}}
\newcommand{\ad}{\mbox{\rm ad}}

\newtheorem{theorem}{Theorem}[section]
\newtheorem{lemma}[theorem]{Lemma}
\newtheorem{proposition}[theorem]{Proposition}
\newtheorem{corollary}[theorem]{Corollary}

\theoremstyle{definition}

\newtheorem{remark}[theorem]{Remark}

\makeatletter
\@namedef{subjclassname@2020}{%
  \textup{2020} Mathematics Subject Classification}
\makeatother

\begin{document}
\title[Rational equivalence for enveloping algebras]{\bf Rational equivalence for enveloping algebras of three-dimensional Lie algebras}

\author{Jacques Alev}
\address[J.~Alev]{Universit\'e de Reims, Laboratoire de Math\'ematiques (UMR 9008 - CNRS),
Moulin de la Housse, B.P. 1039, 51687 Reims cedex 2 (France)}
\email{jacques.alev@univ-reims.fr}
\author{Fran\c cois Dumas}
\address[F.~Dumas]{Universit\'e Clermont Auvergne, CNRS, LMBP, F-63000 Clermont-Ferrand, France }
\email{Francois.Dumas@uca.fr}
\author{C\'esar Lecoutre}
\address[C.~Lecoutre]{Universit\'e Clermont Auvergne, CNRS, LMBP, F-63000 Clermont-Ferrand, France }
\email{Cesar.Lecoutre@uca.fr}

\begin{abstract}{We study from the point of view of rational equivalence the enveloping algebras of Lie algebras of dimension 3 whose derived Lie subalgebra is of dimension 2, 
over an algebraically closed  base field in arbitrary characteristics.}\end{abstract}

\date\today

\maketitle

\section*{Introduction}This article deals with the rational equivalence of enveloping algebras of finite-dimensional Lie algebras, 
that is, their classification up to isomorphism of their skewfields of fractions. 
A structuring axis in the matter is the property of Gelfand-Kirillov initiated in the fundamental article \cite{GK} and since then the subject of numerous developments. 
Our study concerns two families of solvable Lie algebras of dimension three, non necessarily algebraic, in any characteristic. 
Following the classification by \cite{Jac79}, they fall under the case of 3-dimensional  Lie algebras  with derived Lie subalgebra of dimension 2.  
The first family denoted here by $\f g_\alpha$ is paramatrized by the non zero elements of the base field $\kk$ (with $\f g_\alpha$ isomorphic to $\f g_\beta$ if and only if $\beta=\alpha^{\pm 1}$). 
The second one reduces by isomorphism to only one Lie algebra denoted by $\f q$ in this article.

When $\kk$ is algebraically closed, they are the only 3-dimensional non abelian Lie algebras besides $\f {sl}(2)$, 
the Heisenberg algebra and the central extension of the non abelian 2-dimensional Lie algebra. 
It is well known and easy to prove that the enveloping algebras of these three classical examples are rationnally equivalent 
to a Weyl algebra over a purely transcendental extension of $\kk$ and thus they satisfy the Gelfand-Kirillov property. 
Our main goal here is to complete the picture by describing the enveloping skewfields of $\f g_\alpha$ and $\f q$.

The first section is devoted to recalling the context and some useful results. The family of Lie algebras $\f g_\alpha$ is the subject of sections 2 and 3. 
We prove that the Lie algebra $\f g_\alpha$ satisfies the Gelfand-Kirillov property if and only if $\alpha$ lies in the prime subfield $\kk_0$ of $\kk$ 
(Corollary \ref{isoaq}, Theorems \ref{GKl} and \ref{GKlbis}).
The characteristic zero case appears as an example in \cite{GK} and falls within the framework of the proofs of the Gelfand-Kirillov 
conjecture for solvable Lie algebras by \cite{BGR}, \cite{Joseph} or \cite{McConnell}.  
The case where $\kk$ is of characteristic $\ell >0$ is based on a detailed study of the structure of the enveloping skewfield $K(\f g_\alpha)$ over its center $C(\f g_\alpha)$. For $\alpha\notin\FF_\ell$, 
we describe $K(\f g_\alpha)$ as a tensor product over $C(\f g_\alpha)$ of two Weyl skewfields and deduce that the class of $K(\f g_\alpha)$ is of order $\ell$ in the Brauer group of $C(\f g_\alpha)$.

Finding a condition for $K(\f g_\alpha)$ and $K(\f g_\beta)$ to be isomorphic remains only when $\alpha$ and~$\beta$ are not in $\kk_0$. 
We show (Proposition \ref{CS}) that a sufficient condition for such an isomorphism is that $\alpha$ and $\beta$ belong to the same orbit under the homographic action of ${\rm GL}_2(\Z)$ on~$\kk\setminus\kk_0$, and give
examples of application for some finite fields.
We do not know whether or not this condition is necessary in all generality.
The goal of section 3 is to prove (Theorem \ref{mainthm}) that this is the case in characteristic zero if we only consider isomorphisms of valued skewfields
for the valuation canonically associated (in the sense of \cite{KT}) to the derived Lie subalgebra.
In the case where $\kk=\C$ we obtain a complete classification up to valued isomorphism with arithmetical interpretations of the orbits (Corollary \ref{corell}).

Section 4 is devoted to corresponding results for the enveloping skewfield $K(\f q)$ of the Lie algebra $\f q$, 
which does not satisfy the Gelfand-Kirillov property, and to its separation from the skewfields $K(\f g_\alpha)$ 
up to valued isomorphism.

\section{Preliminary results}
\subsection{Classification of three-dimensional Lie algebras}\label{Jac}
Let $\kk$ be an algebraically closed field. The classification by \cite{Jac79} up to isomorphism
of Lie algebras $\f  g$ of dimension 3 over $\kk$ according to the dimension $d$ of the derived Lie subalgebra $\f  g'$  
leads to the following classical examples:  $\f  g$ abelian (if $d=0$),   $\f  g=\f h$ the Heisenberg algebra or $\f  g=\f b$ a 
central extension of the two-dimensional nonabelian Lie algebra (if $d=1$), $\f  g=\f {sl}(2)$   (if $d=3$). The remaining situation $d=2$ splits into the following two cases, which are the object of our study:

(i) A family of Lie algebras $\f {g}_\alpha$ indexed by a nonzero scalar $\alpha$ and whose Lie brackets on a basis $\{x,y,z\}$ are given by
\begin{equation}\label{brakg}[x,y]=y,\quad [x,z]=\alpha z,\quad [y,z]=0.\end{equation}
They satisfy $\f {g}_\alpha\cong\f {g}_\beta$ if and only if $\alpha=\beta$ or $\alpha=\beta^{-1}$.

(ii) A family of Lie algebras $\f {q}_\gamma$ indexed by a nonzero scalar $\gamma$ and whose Lie brackets on a basis $\{x,y,z\}$ are given by
\[[x,y]=y,\quad [x,z]=z+\gamma y,\quad [y,z]=0.\]
It turns out by the change of basis $\{x,y,\gamma^{-1} z\}$ that this family reduces to a single algebra 
$\f {q}=\f {q}_1$, with brackets
\begin{equation}\label{brakq}[x,y]=y,\quad [x,z]=y+z,\quad [y,z]=0\end{equation}

\subsection{Rational equivalence of enveloping algebras}\label{prelimGK}
For Lie algebras of dimension 3, the problem of classification up to isomorphism of the enveloping algebra is solved by \cite{CKL} : 
assuming that $\car\kk\neq2$, two Lie algebras of dimension 3 over $\kk$ are isomorphic if and only if their enveloping algebras are isomorphic as associative $\kk$-algebras. 

The problem is more delicate if we consider rational equivalence, that is the classification of Lie algebras up to isomorphism of their enveloping skewfields. 
Let us recall that the enveloping algebra of any finite dimensional Lie algebra $\f {g}$ is a noetherian domain and thus admits a skewfield of fractions. 
According to the notations of \cite{Dix}, we denote by $U(\f  g)$ the enveloping algebra, by $K(\f {g})$ its skewfield of fractions and by $C(\f  g)$ the center of $K(\f g)$. 

A main argument on this topic comes from the seminal paper \cite{GK}:
a Lie algebra $\f  g$ over $\kk$ is said to satisfy the Gelfand-Kirillov property if there exist integers $n\geqslant 1,m\geqslant 0$
such that $K(\f  g)$ is $\kk$-isomorphic to a Weyl skewfield $\mathcal D_{n,m}(\kk)$, where $\mathcal D_{n,m}(\kk)$ is the skewfield of fractions 
of a Weyl algebra $A_n(L)$ over a commutative field $L$ which is a rational function field of transcendence degree $m$ over $\kk$. 
The literature on this subject is extremely rich (see for instance references in \cite{AOV} or \cite{Premet}).
Let us simply observe here that, for obvious reasons of transcendence degree, a 3-dimensional non-abelian Lie algebra 
$\f g$ satisfies the Gelfand-Kirillov property if and only $K(\f g)$ is $\kk$-isomorphic to a Weyl field $\mathcal D_{1,1}(\kk)$. Explicitly the Weyl algebra $A_1(L)$ is 
the associative algebra generated over a field $L$ by two elements $p$ and $q$ satisfying the commutation relation $pq-qp=1$, and $\mathcal D_{1,1}(\kk)$ is 
the skewfield of fractions of $A_1(L)$ when $L=\kk(t)$ is a purely transcendental extension of~$\kk$. 

Back to the classification \ref{Jac}, it is well known that the Gelfand-Kirillov property is satisfied in the three cases of $\f h$, $\f b$ and $\f {sl}(2)$. 
On the contrary the Lie algebra $\f g_\alpha$ when $\kk$ is of characteristic zero and $\alpha$ is not a rational 
number appears in the original paper \cite{GK} as a typical example of a non algebraic Lie algebra which does not satisfy the Gelfand-Kirillov property. 

\subsection{Discrete valuations on the enveloping skewfields}\label{DV}The authors of \cite{KT} define for any Lie subalgebra $\f  a$ of a finite dimensional Lie algebra~$\f  g$ 
a degree function on $U(\f  g)$ canonically associated to $\f  a$. 
Denoting by $\f  c$ a direct summand of $\f  a$ in $\f  g$ and by $y_1,\ldots,y_k$ a basis of $\f  c$, 
any element $u$ of $U(\f  g)$ can be written uniquely as a finite sum $u=\sum_{m}a_my_1^{m_1}\cdots y_k^{m_k}$ with $m=(m_1,\ldots,m_k)\in\Z_{\geqslant 0}^k$ and $a_m\in\mathcal U(\f a)$ (see \cite{Dix} proposition 2.2.7).
We set $\deg u=\max_{a_m\not=0}|m|$ where $|m|=m_1+\cdots+m_k$. 

It is relevant within the framework of the classification considered in \ref{Jac} to choose for $\f a$ the derived subalgebra of $\f g$. 
In the case of the algebras $\f  g_\alpha$ and $\f  q$ studied here, $\f  a=\kk y\oplus\kk z$ is also the only abelian ideal of codimension~1. We have $U(\f a)=\kk[y,z]$ then $\deg(y)=\deg(z)=0$, 
and $\deg(x)=1$ by taking $\f  c=\kk x$. 
We introduce the discrete valuation $v=-\deg$ on $U(\f g_\alpha)$ and on $U(\f q)$. 
Its canonical extension provides the enveloping skewfields $K(\f g)$ and $K(\f q)$ with a structure of valued skewfields.


\section{Enveloping skewfields of the Lie algebras $\f {g}_\alpha$}
\label{galpha}
\subsection{Notations}\label{notgalpha}
In this section, $\kk$ is a field.
For any $\alpha\in\kk^\times$, we consider the Lie algebra $\f g_\alpha$ over $\kk$ whose brackets on a basis $\{x,y,z\}$ are defined by relations \eqref{brakg}. 
By the Poincar\'e-Birkhoff-Witt Theorem, the enveloping algebra $U(\f g_\alpha)$ is the associative $\kk$-algebra generated by three generators  $x,y,z$ and relations
\begin{equation}\label{Ugrel}yz=zy,\quad xy-yx=y,\quad xz-zx=\alpha z.\end{equation}
It can be viewed as the iterated Ore extension
\begin{equation}\label{UgOre}U(\f g_\alpha)=\kk[y,z][x\,;\,D_\alpha],\quad \text{where } D_\alpha=y\partial_y+\alpha z\partial_z.\end{equation}
This means that any element $U(\f g_\alpha)$ can be  written uniquely as a finite sum $\sum_{i=0}^nf_i(y,z)x^i$ with $f_i\in\kk[y,z]$, and 
\begin{equation}\label{UgOre2}xf=fx+D_\alpha(f) \text{ for any } f\in\kk[y,z].\end{equation}
Denoting again by $D_\alpha$ the extension of $D_\alpha$ to a $\kk$-derivation of the field $\kk(y,z)$, we can embed $U(\f g_\alpha)$ in the algebra $U'(\f g_\alpha)=\kk(y,z)[x\,;\,D_\alpha]$, 
and their common skewfield of fractions is classically denoted by
\begin{equation}\label{UgK}K(\f g_\alpha)=\kk(y,z)(x\,;\,D_\alpha).
\end{equation}
The description of $K(\f g_\alpha)$ splits into two quite different studies depending on whether $\kk$ is of zero or prime characteristic.

\subsection{Center and Gelfand-Kirillov property in zero characteristic}\label{zc}
We suppose in this section that $\kk$ is of characteristic zero and identify the prime subfield of $\kk$ with $\Q$.
We fix a non zero element $\alpha$ in $\kk$.

\begin{proposition}Let $C(\f g_\alpha)$ be the center of $K(\f{g}_\alpha)$. 
If $\alpha\in\kk\setminus\Q$, then $C(\f g_\alpha)=\kk$. 
If $\alpha\in\Q$, then $C(\f{g}_\alpha)=\kk(y^{p}z^{-q})$ where $\alpha=p/q$ with $p,q\in\Z$, $q\not=0$ and $\gcd(p,q)=1$.
\end{proposition}

\begin{proof}By \eqref{UgK} we deduce from Theorem 5.6 of \cite{Goo} that $C(\f g_\alpha)$ is the kernel of the derivation $D_\alpha$ of $\kk(y,z)$.
Any element $f\in\kk(y,z)$ can be expanded in the  extension $\kk(\!(y)\!)(\!(z)\!)$ as a series
$f=\sum_{j\geqslant j_0}(\sum_{i\geqslant i_j}\lambda_{ij}y^i)z^j$. Applying $D_\alpha=y\partial_y+\alpha z\partial_z$ it follows that
$f\in C(\f g_\alpha)$ if and only if $i+\alpha j=0$ for any couple $(i,j)$ in the support of $f$.
If $\alpha\notin\Q$, this forces $i=j=0$ thus $f\in\kk$.
Suppose now that  $\alpha=p/q$ with $p,q\in\Z$, $q\not=0$ and $\gcd(p,q)=1$. 
In particular $z'=y^{p}z^{-q}$ satisfies $D_\alpha(z')=0$.
Since $\gcd(p,q)=1$ there exists $u,v\in\Z$ such that $pu+qv=1$.
Setting $y'=y^{v}z^{u}$, we have $D_\alpha(y')=\lambda y'$ with $\lambda=v+\alpha u\in\kk^\times$ and $\kk(y,z)=\kk(y',z')$ by birational change of variables.
It follows that $C(\f{g}_\alpha)=\kk(z')$.
\end{proof}

\begin{corollary}
\label{isoaq}
The Lie algebra $\f{g}_\alpha$ satisfies the Gelfand-Kirillov property if and only if $\alpha~\in~\Q$. In this case $K(\f{g}_\alpha)$
is $\kk$-isomorphic to $\mathcal D_{1,1}(\kk)$.
\end{corollary}

\begin{proof}If $\alpha\notin\Q$, the skewfield $K(\f g_\alpha)$ is of dimension 3 over its center $\kk$ in the sense of \cite{GK}  
and then cannot be isomorphic to a Weyl skewfield. If $\alpha\in\Q$ we use the notations of the previous proof to deduce that $K(\f g_\alpha)=\kk(y',z')(x\,;\,D_\alpha)$ with relations
$[x,z']=[y',z']=0$ and $[x,y']=\lambda y'$. We set $x'=\lambda^{-1}{y'}^{-1}x$ to conclude that $K(\f g_\alpha)$ is the skewfield generated over $\kk$ by $x',y',z'$ with relations $[x',y']=1$ and $[x',z']=[y',z']=0$.
\end{proof}

\subsection{Center and Gelfand-Kirillov property in prime characteristic}\label{CGKell}
We suppose in this section that $\kk$ is of characteristic $\ell >0$ and identify the prime subfield of $\kk$ with the finite field $\FF_\ell$. We fix a non zero element~$\alpha$ in~$\kk$.
As in \ref{zc}, the situation is quite different depending on whether or not~$\alpha$ lies in the prime subfield.

\begin{theorem}\label{GKl}
We suppose that $\alpha\in\FF_\ell^\times$. Then:
\begin{itemize}
 \item[(i)] The center $C(\f g_\alpha)$ of $K(\f{g}_\alpha)$ is a rational function field of transcendence degree 3 over $\kk$.
More precisely $C(\f{g}_\alpha)=\kk(x^\ell-x,y^\ell,z')$, where $z'=y^{-a}z$ for $1\leqslant a\leqslant \ell-1$ such that $\overline a=\alpha$. 
\item[(ii)] The Lie algebra $\f{g}_\alpha$ satisfies the Gelfand-Kirillov property. More precisely $K(\f{g}_\alpha)$ is $\kk$-isomorphic to $\mathcal D_{1,1}(\kk)$ and hence is of dimension $\ell^2$ over its center.
\end{itemize}
\end{theorem}

\begin{proof}It follows from \eqref{UgOre2} and \eqref{UgK} that $\ker D_\alpha\subset C(\f g_\alpha)$.
Therefore $y^\ell$ and~$z^\ell$ are elements of $C(\f g_\alpha)$. 
Thanks to relations \eqref{Ugrel} we have for all $i\in\Z$
\begin{equation}
\label{relx}
x^iy=y(x+1)^i \quad\mbox{and}\quad  x^iz=z(x+\alpha)^i.
\end{equation}
In particular
\begin{equation}
\label{deg l}
(x^\ell-x)y=y(x^\ell-x)\ \ \text{and}\ \ (x^\ell-\alpha^{\ell-1}x)z=z(x^\ell-\alpha^{\ell-1}x).
\end{equation}
Since $\alpha\in\FF_\ell^\times$, we have $\alpha^{\ell-1}=1$ and the element $x^{\ell}-x=x^\ell-\alpha^{\ell-1}x$ commutes with both $y$ and $z$. 
Thus the field $C_\ell=\kk(x^\ell-x,y^\ell,z^\ell)$ is contained in  $C(\f g_\alpha)$. 
Since the dimension of $K(\f g_\alpha)$ over its center is a square and the dimension of $K(\f g_\alpha)$ over $C_\ell$ is $\ell^3$, then the dimension of $K(\f g_\alpha)$ over $C(\f g_\alpha)$ is necessarily $\ell^2$.
The element $z'=y^{-a}z\in\kk(y,z)$ satisfies
\[D(z')=D_\alpha(y^{-a})z+y^{-a}D_\alpha(z)=-\overline{a}D_\alpha(y)y^{-a-1}z+\alpha y^{-a}z=0,\]
hence $z'\in C(\f g_\alpha)$.
Since $z'\notin \kk(x^\ell-x,y^\ell,z^\ell)$ we deduce by a degree argument that $C(\f g_\alpha)=\kk(x^\ell-x,y^\ell,z^\ell,z')$ and conclude that $C(\f g_\alpha)=\kk(x^\ell-x,y^\ell,z')$ because $z^\ell=(z')^\ell (y^\ell)^{a}$.
Finally $K(\f g_\alpha)$ is generated by $x,y$ and $z'$ with relations $[x,y]=y$ and $[x,z']=[y,z']=0$.
Setting $x'=xy^{-1}$ we obtain $[x',y]=1$ and $[x',z']=[y,z']=0$ and the proof is complete.
\end{proof}

The rest of this section is devoted to the case where $\alpha\notin\FF_\ell$. We start by the following lemma  which introduces a central element.

\begin{lemma}\label{ccentral}We suppose that $\alpha\notin\FF_\ell$.
Let $\mu=(\alpha^\ell-\alpha)^{\ell-1}$ and $\lambda=-\mu-1$. Then the element
$$c=x^{\ell^2}+\lambda x^\ell+\mu x =(x^{\ell}-x)^{\ell}-\mu(x^{\ell}-x)$$ of $U(\f g_\alpha)$ is central in  $K(\f{g}_\alpha)$.
\end{lemma}
\begin{proof} Using relations \eqref{relx} we have
\begin{align*}
cy&=x^{\ell^2}y+\lambda x^\ell y+\mu xy=y\big((x+1)^{\ell^2}+\lambda(x+1)^{\ell}+\mu(x+1)\big)\\
&=y(x^{\ell^2}+1+\lambda x^\ell+\lambda+\mu x+\mu)=yc+y(1+\lambda+\mu)=yc,
\end{align*}
and
\begin{align*}
cz&=x^{\ell^2}z+\lambda x^\ell z+\mu xz=z\big((x+\alpha)^{\ell^2}+\lambda(x+\alpha)^{\ell}+\mu(x+\alpha)\big)\\
&=zc+z\big(\alpha^{\ell^2}+\lambda \alpha^\ell+\mu \alpha)=zc+z\big(\alpha^{\ell^2}-\alpha^\ell-\mu(\alpha^\ell-\alpha)\big)=zc.
\end{align*}
as desired.
\end{proof}

\begin{remark}
The element $c$ can also be obtained as an element of $\kk(x)$ that is invariant under both actions of $y$ and $z$.
Since $yx=(x-1)y$ and $zx=(x-\alpha)z$,
these actions correspond to automorphisms $\sigma_1$ and $\sigma_\alpha$ of $\kk(x)$ where $\sigma_\gamma$ is the $\kk$-automorphism of $\kk(x)$ defined for any nonzero scalar $\gamma\in\kk$ by $\sigma_\gamma(x)=x-\gamma$.
It is a classical fact from modular invariant theory (see Theorem 1.11.2 from \cite{CW} for the homogenized case) that $\kk(x)^{\sigma_\gamma}=\kk(t_\gamma)$ where $t_\gamma=\prod_{i\in\FF_\ell}\sigma^i(x)=x^\ell-\gamma^{\ell-1}x$.
Indeed we have $\kk(t_\gamma)\subset \kk(x)^{\sigma_\gamma}\subset \kk(x)$ and we conclude by a degree argument since the extension $\kk(t_\gamma)\subset\kk(x)$ is of prime degree $\ell$. 
Applying this result successively to $\sigma_1$ and $\sigma_\alpha$, we obtain $(\kk(x)^{\sigma_1})^{\sigma_\alpha}=\kk(t_1)^{\sigma_\alpha}=\kk(c)$ for $t_1=x^\ell-x$, $\sigma_\alpha(t_1)=t_1-(\alpha^\ell-\alpha)^{\ell-1}$ and 
$c=t_1^\ell-(\alpha^\ell-\alpha)^{\ell-1}t_1$, which is the element defined in Lemma \ref{ccentral}.
\end{remark}

Later on we will need the two following classical arguments on rationality in commutative Laurent series. 
For the convenience of the reader, we give in the following lemma a formulation and proof adapted to our context.

\begin{lemma}\label{lemmecomm}{ \ }
\begin{itemize}
 \item[(i)]\label{Bbkibis}
Let $K \subset L$ be a field extension.
Then we have $K(\!(X)\!)\cap L(X)=K(X)$.
\item[(ii)]\label{intersection} Let $K$ be a field of positive characteristic $\ell>0$.
Then we have $K(X)\cap K(\!(X^\ell)\!)=K(X^\ell)$.
 \end{itemize}
\end{lemma}

\begin{proof}
The proof of assertion (i)  is a direct consequence of the analogous property $K[[X]]\cap L(X)=K[[X]]\cap K(X)$ for power series, see \cite[\S 5.2]{YA} and \cite[\S 4, Exercice 1]{Bbki}. 

In assertion (ii) it is clear that $K(X^\ell)\subset K(X)\cap K(\!(X^\ell)\!)$.
For the reverse inclusion, let $F=\sum_{i\geqslant i_0} a_i X^{\ell i}$ be an element of $K(\!(X^\ell)\!)$ where $a_i\in K$ for all $i\geqslant i_0$ and suppose that $F\in K(X)$. 
There exist relatively prime polynomials $P,Q\in K[X]$ such that $F=PQ^{-1}$.
Denoting by $d$ the usual derivation $d/dX$ in $\kk(\!(X)\!)$ we have $d(F)=0$ because $F\in K(\!(X^\ell)\!)$, thus $d(P)Q-Pd(Q)=0$ in $K[X]$. This implies that $P$ divides $d(P)$ in $K[X]$, then $d(P)=0$ and $d(Q)=0$.
It is easy to check that $\ker d\cap K[X]=K[X^\ell]$.
We deduce that $P\in\kk[X^\ell]$ and $Q\in\kk[X^\ell]$ hence $F\in K(X^\ell)$.
\end{proof}

\begin{lemma}We suppose that $\alpha\notin\FF_\ell$.
\begin{itemize}
 \item[(i)] \label{kernel}
The kernel of the derivation $D_\alpha$ of $\kk(y,z)$ is equal to $\kk(y^\ell,z^\ell)$.
 \item[(ii)]\label{cent}
The centralizer of $x$ in ${K(\f{g}_\alpha)}$ is the commutative subfield $\kk(x,y^\ell,z^\ell)$.
In particular it has codimension $\ell^2$ in $K(\f g_\alpha)$.
\end{itemize}
\end{lemma}

\begin{proof}
It is clear that $\kk(y^\ell,z^\ell)\subset\ker{D_\alpha}$.
To prove the reverse inclusion we set $f$ an element in $\kk(y,z)$ such that $D_\alpha(f)=0$.
Embedding $\kk(y,z)$ in $\kk(y)(\!(z)\!)$, 
we introduce its expansion $f=\sum_{i\geqslant v}a_iz^i$ with $v\in\Z$ and $a_i\in\kk(y)$ for any $i\geqslant v$.
Then $D_\alpha(f)=\sum_{i\geqslant v}\left(D_\alpha(a_i)+i\alpha a_i\right)z^i$.
Thus we have $D_\alpha(a_i)+i\alpha a_i=0$ for any $i\geqslant v$.
Considering now the embedding of $\kk(y)$ in $\kk(\!(y)\!)$, we expand 
$a_i=\sum_{k\geqslant w_i} \beta_{i,k} y^k$ with $w_i\in\Z$ and $\beta_{i,k}\in\kk$ for any $k\geqslant w_i$.
We compute $D_\alpha(a_i)+i\alpha a_i=\sum_{k\geqslant w_i}(k+i\alpha)\beta_{i,k}y^k$ and deduce that $k+i\alpha=0$ for all $i,k$ such that $\beta_{i,k}\neq0$.
Since $\alpha\notin\FF_\ell$, this implies that $\ell$  divides both~$i$ and $k$ for all $i,k$ such that $\beta_{i,k}\neq0$.
Therefore $a_i\in\kk(\!(y^\ell)\!)\cap \kk(y)=\kk(y^\ell)$ for any $i\geqslant v$ thanks to assertion (ii) of Lemma \ref{intersection}. We conclude that $f$ is an element of $\kk(y^\ell)(\!(z^\ell)\!)$.
Let us denote $K=\kk(y^\ell)$. We have $f\in K(\!(z^\ell)\!)\subset K(\!(z)\!)$. 
Recalling that $f\in\kk(y,z)=\kk(y)(z)$, assertion (i) of Lemma \ref{Bbkibis} with $L=\kk(y)$ implies that $f\in K(\!(z)\!)\cap \kk(y)(z)=K(z)$. It follows again from Lemma \ref{intersection} that $f\in K(\!(z^\ell)\!)\cap K(z)=K(z^\ell)=\kk(y^\ell,z^\ell)$ which ends the proof of assertion (i).

Assertion (ii) is then a particular case of \cite[Theorem 5.8]{Goo}.
\end{proof}

\begin{theorem}\label{GKlbis}We suppose that $\alpha\notin\FF_\ell$.
\begin{itemize}
 \item[(i)] \label{zell}
The center $C(\f g_\alpha)$ of $K(\f g_\alpha)$ is equal to $\kk(y^\ell,z^\ell,c)$.
In particular, $K(\f {g}_\alpha)$ is of dimension $\ell^4$ over its center.
 \item[(ii)]\label{Da}
The Lie algebra $\f{g}_\alpha$ does not satisfy the Gelfand-Kirillov property.
\end{itemize}
\end{theorem}

\begin{proof}
By Lemma \ref{ccentral}, $C_\ell=\kk(y^\ell,z^\ell,c)$ is a subfield of $C(\f {g}_\alpha)$.
Hence the dimension of $K(\f g_\alpha)$ over $C(\f {g}_\alpha)$ is a square that divides $\ell^4$.
Moreover $C(\f {g}_\alpha)$ is contained in the centralizer of $x$ in $K(\f g_\alpha)$ which is of codimension $\ell^2$ in $K(\f g_\alpha)$ thanks to Lemma \ref{cent}, and the inclusion is strict because $x\notin C(\f {g}_\alpha)$. 
Thus the only possibility is that $K(\f g_\alpha)$ is of dimension $\ell^4$ over $C(\f {g}_\alpha)=C_\ell$.
Assertion (ii) follows since the dimension of a Weyl skewfield $\mathcal{D}_{1,1}(\kk)$ over its center is $\ell^2$ (see for instance \cite[Proposition 1.1.3]{Bois}).\end{proof}

The following proposition describes the structure of $K(\f g_\alpha)$ in relation to Weyl skewfields.

\begin{proposition}
\label{struc1}
We suppose that $\alpha\notin\FF_\ell$.
Let $L$ be the skewfield generated by $z$, $y^\ell$ and $x^\ell-x$ in $K(\f g_\alpha)$.
\begin{itemize}
 \item[(i)] The center of $L$ is $C(\f g_\alpha)$.
 \item[(ii)] The centralizer of $L$ in $K(\f{g}_\alpha)$ is the skewfield $L'$ generated by $y$, $z^\ell$ and $x^\ell-\alpha^{\ell-1}x$.
 \item[(iii)] $L$ and $L'$ are both $\kk$-isomorphic to a Weyl skewfield $\mathcal D_{1,1}(\kk)$.
 \item[(iv)]\label{thmDCT}
$K(\f {g}_\alpha)$ is $C(\f g_\alpha)$-isomorphic to the tensor product of $L$ and $L'$ over $C(\f g_\alpha)$.
\end{itemize}
\end{proposition}

\begin{proof}
(i) Observe that $\kk(c,y^\ell,z)$ is contained in the centralizer $\mathcal{C}_L(z)$ of $z$ in $L$. 
Since $x^\ell-x$ does not commute with $z$, $\mathcal{C}_L(z)\neq L$ and by a degree argument we obtain $\mathcal{C}_L(z)=\kk(c,y^\ell,z)$.
Similarly we show that the centralizer  of $x^\ell-x$ in $L$ is $\mathcal{C}_L(x^\ell-x)=\kk(x^\ell-x,y^\ell,z^\ell)$. 
We deduce that the center $C(L)$ of $L$ satisfies $C(L)\subset \mathcal{C}_L(z)\cap \mathcal{C}_L(x^\ell-x)=\kk(c,y^\ell,z^\ell)$ by Lemma \ref{ccentral}. 
Then $C(L)\subset C(\f g_\alpha)$ by Theorem \ref{zell}. Since $C(\f g_\alpha)\subset L$ , we conclude that $C(L)=C(\f g_\alpha)$.

(ii) Denote by $L'$ the centralizer of $L$ in $K(\f g_\alpha)$. 
Since $L$ is a simple subalgebra of $K(\f g_\alpha)$ over $C(\f g_\alpha)$ of dimension $\ell^2$, it follows from the Double Centralizer Theorem \cite[Theorem 12.7, assertion (ii)]{Pierce} 
that $L'$ is also of dimension $\ell^2$ over $C(\f g_\alpha)$. The skewfield generated in $K(\f g_\alpha)$ by $y$, $z^\ell$ and $x^\ell-\alpha^{\ell-1}x$ being clearly contained in $L'$, equality again follows by a degree argument.

(iii)
In $L$ we have $[x^\ell-x,z]=(\alpha^\ell-\alpha)z\neq 0$ and $y^\ell$ is central.
By setting $t'=(x^\ell-x)(\alpha^\ell-\alpha)^{-1}z^{-1}$ we obtain three generators $t',z$, $y^\ell$ of $L$ satisfying the relations $[t',z]=1$ and $[t',y^\ell]=[z,y^\ell]=0$.
Thus $L$ is $\kk$-isomorphic to the Weyl skewfield $\mathcal{D}_{1,1}(\kk)$.
We obtain similarly that $L'$ is $\kk$-isomorphic to $\mathcal{D}_{1,1}(\kk)$ since $[x^\ell-\alpha^{\ell-1}x,y]=(1-\alpha^{\ell-1})y\neq 0$.

(iv) It is a direct application of the Double Centralizer Theorem, see for instance \cite[Theorem 12.7, assertion (iv)]{Pierce}. 
\end{proof}

\begin{corollary}\label{brauer}
The class $[K(\f g_\alpha)]$ of $K(\f g_\alpha)$ in the Brauer group of $C(\f g_\alpha)$ is of order $\ell$.
\end{corollary}
\begin{proof}
Both $L$ and $L'$ are of dimension $\ell^2$ over $C(\f g_\alpha)$.
The order $o([L])$ of $[L]$ in the Brauer group of $C(\f g_\alpha)$ divides $\sqrt{\ell^2}=\ell$
(see for instance \cite[Theorem 4.4.5]{H}).
Since $[L]$ is not trivial, hence $o([L])=\ell$.
Similarly $o([L'])=\ell$.
Moreover we have $[K(\f g_\alpha)]=[L][L']$ by the previous proposition.
Therefore $o([K(\f g_\alpha)])=\ell$.
\end{proof}

\subsection{Rational isomorphisms and embeddings}\label{plongements}
We suppose in this section that $\kk$ is an arbitrary field.  
It follows from Corollary \ref{isoaq} and Theorem \ref{GKl} that $K(\f g_\alpha)$ and~$K(\f g_\beta)$ are isomorphic 
if $\alpha$ and $\beta$ lie in the prime subfield $\kk_0$ of $\kk$.
Moreover $K(\f g_\alpha)$ and $K(\f g_\beta)$ are not isomorphic when $\alpha\in\kk_0$ and $\beta\notin\kk_0$ by Theorem \ref{GKlbis}.
Hence we focus here on the situation where $\alpha\notin\kk_0$ and $\beta\notin\kk_0$.
We use the notation $x,y,z$ for generators of $K(\f g_\alpha)$ as in \eqref{UgK}, and similarly $x',y',z'$ for generators of $K(\f g_\beta)$.

\begin{lemma}
\label{embedmonomial}
Let $\alpha\in\kk\setminus\kk_0$.
Let $M=\left(\begin{smallmatrix}n &q\\m & r\end{smallmatrix} \right)\in
M_2(\Z)$ be a matrix whose determinant is nonzero in $\kk$ and set
$\beta=\frac{n \alpha +q}{m \alpha +r}$.
Then there exists an injective $\kk$-algebra morphism $\varphi :
K(\f {g}_{\beta})\rightarrow K(\f {g}_{\alpha})$ defined by
\begin{align*}
\varphi(x')=\textstyle\frac{1}{m\alpha + r} x,\quad
\varphi(y')=y^{r}z^{m}\quad\mbox{and}\quad
\varphi(z')=y^{q}z^{n}.\end{align*}
Moreover if $M$ is invertible over $\Z$ then $\varphi$ is an isomorphism.
\end{lemma}

\begin{proof}
The linear map $\varphi:\f{g}_\beta\rightarrow K(\f {g}_{\alpha})$
given by the formulae above
satisfies
\begin{align*}
&\textstyle[\varphi(y'),\varphi(z')]=0=\varphi([y',z'])\\
&\textstyle[\varphi(x'),\varphi(y')]=\frac{1}{m\alpha +
r}[x,y^{r}z^{m}]=\frac{1}{m\alpha +
r}(r+m\alpha)y^{r}z^{m}=\varphi(y')=\varphi([x',y']),\\
&\textstyle[\varphi(x'),\varphi(z')]=\frac{1}{m\alpha +
r}[x,y^{q}z^{n}]=\frac{1}{m\alpha +
r}(q+n\alpha)y^{q}z^{n}=\beta\varphi(z')=\varphi([x',z']).
\end{align*}
By universal properties of the enveloping algebra $U(\f{g}_\beta)$ and
of localization, we obtain an algebra morphism
$\varphi:K(\f{g}_\beta)\to K(\f{g}_\alpha)$.
Note that the map $\varphi$ is injective since $y^{r}z^{m}$ and
$y^{q}z^{n}$ are algebraically independent thanks to the hypothesis
$\det M\neq 0$ in $\kk$
(use the Jacobian criterion in zero characteristic or its adaptation
\cite[Theorem 1]{MSS} to positive characteristic).
Finally, if $M$ is invertible over $\Z$, we can define similarly the map
$\varphi^{-1}$ by the images of $x,y,z$ as monomials in
$x',y',z'$.\end{proof}

\begin{proposition}\label{CS}Let $\alpha,\beta\in\kk\setminus\kk_0$. If $\alpha$ and $\beta$ are in the same orbit under the homographic action of ${\rm GL}_2(\Z)$ on $\kk\setminus\kk_0$, then $K(\f {g}_{\alpha})$ 
and $K(\f g_\beta)$ are $\kk$-isomorphic.
\end{proposition}
\begin{proof}We apply Lemma \ref{embedmonomial} considering the homographic action of the group $\GL_2(\Z)$ on $\kk\setminus\kk_0$
defined by $\left(\begin{smallmatrix}n &q\\m & r\end{smallmatrix} \right)\cdot\alpha=\frac{n\alpha+q}{m\alpha+r}$.
\end{proof}

\begin{remark}\label{ell}
We take here $\kk=\C$. For any $\alpha\in\R\setminus\Q$, the orbit of $\alpha$ under the above action of $\GL_2(\Z)$ is the set
of real numbers having the same continued fraction development after a certain point; see \cite[Theorem 175]{HW}.

If $\alpha\in\C\setminus\R$, the sign of the imaginary part of $g\cdot\alpha$ is the product of the sign of the imaginary part of $\alpha$ by $\det g$ for any $g\in\GL_2(\Z)$.
We deduce a bijective correspondence between the orbits of the elements of $\C\setminus\R$ under the action of $\GL_2(\Z)$, and the orbits of the elements of the Poincar\'e halfplane $\HH$ under the action of the subgroup $\SL_2(\Z)$.
By classical number theoretical results (see \cite{JPS} pp 81-82), the quotient set $\HH/\SL_2(\Z)$ can be identified with the set of isomorphism classes of complex elliptic curves. Observe that in the associated correspondence 
with lattices of $\C$, the lattice $\Z\oplus\alpha\Z$ is by \eqref{UgOre} the set of the eigenvalues of the inner derivation $\ad_x$ in the localization $\kk[y^\pm,z^\pm][x\,;\,D_\alpha]$ of $U(\mathfrak g_\alpha)$.
\end{remark}

In positive characteristic, besides the kind of morphisms introduced in Lemma \ref{embedmonomial}, any $K(\f g_\beta)$ can be embedded in $K(\f g_\alpha)$ for any $\alpha\in\kk\setminus\FF_\ell$.

\begin{proposition}
Assume that $\car\kk=\ell>0$. Let $\alpha\in\kk\setminus\FF_\ell$ and $\beta\in\kk^\times$. 
There exists an injective $\kk$-algebra morphism $\psi : K(\f {g}_{\beta})\rightarrow K(\f {g}_{\alpha})$ defined by
\begin{align*}
\psi(x')=\frac{\beta-\alpha}{\alpha^\ell-\alpha}\,x^\ell+\frac{\alpha^\ell-\beta}{\alpha^\ell-\alpha}\,x,\quad \psi(y')=y\quad\mbox{and}\quad \psi(z')=z,
\end{align*}and its image is of codimension $\ell$.
\end{proposition}

\begin{proof}
It is clear that $[\psi(y'),\psi(z')]=0=\psi([y',z'])$.
Moreover by using relations \eqref{relx} with $i=\ell$ we have
\begin{align*}
[\psi(x'),\psi(y')]&\textstyle=\frac{\beta-\alpha}{\alpha^\ell-\alpha}[x^\ell,y]+\frac{\alpha^\ell-\beta}{\alpha^\ell-\alpha}[x,y]\\
&=\textstyle\left(\frac{\beta-\alpha}{\alpha^\ell-\alpha}+\frac{\alpha^\ell-\beta}{\alpha^\ell-\alpha}\right)y= y=\psi([x',y']),\\
[\psi(x'),\psi(z')]&\textstyle=\frac{\beta-\alpha}{\alpha^\ell-\alpha}[x^\ell,z]+\frac{\alpha^\ell-\beta}{\alpha^\ell-\alpha}[x,z]\\
&=\textstyle\left(\frac{\beta-\alpha}{\alpha^\ell-\alpha}\alpha^\ell+\frac{\alpha^\ell-\beta}{\alpha^\ell-\alpha}\alpha\right)z=\beta z=\psi([x',z']).
\end{align*}
\end{proof}

\subsection{Some examples in positive characteristic}
For $\kk$ of positive characteristic, arithmetical arguments allow to deduce from Proposition \ref{CS} isomorphism results in some particular cases.
\begin{corollary}\label{SL}
Assume that $\car\kk=\ell>2$ with $\FF_{\ell^2}\subset\kk$.
Then $K(\f g_\alpha)$ and $K(\f g_\beta)$ are isomorphic as $\kk$-algebras for all $\alpha,\beta\in\FF_{\ell^2}\setminus\FF_\ell$.
\end{corollary}
\begin{proof}
The proof consists in showing that the action of $\rm{SL}_2(\Z)$ on $\FF_{\ell^2}\setminus\FF_\ell$ is transitive.
The standard congruence map $\rm{SL}_2(\Z)\to\rm{SL}_2(\FF_\ell)$ is onto (see for instance \cite[Lemma 1.38]{Shi} or \cite[Lemma 6.3.10]{HC}).
So thanks to Proposition \ref{CS} it is sufficient to prove that the induced $\rm{SL}_2(\FF_\ell)$-action on $\FF_{\ell^2}\setminus\FF_\ell$ is transitive.

Let $u\in \FF_{\ell}^\times\setminus(\FF_\ell^{\times})^2$ and choose $\theta\in\FF_{\ell^2}^\times$ such that $\theta^2=u$. Then $\{1,\theta\}$ is a $\FF_\ell$-basis of $\FF_{\ell^2}$.
Let us compute the cardinal of the orbit of~$\theta$.
An element $M=\left(\begin{smallmatrix}u_1 & v_1 \\ w_1 & t_1 \end{smallmatrix}\right)$ of $\rm{SL}_2(\FF_\ell)$ lies in $\mbox{Stab}_{\rm{SL}_2(\FF_\ell)}(\theta)$ if and only if 
$\frac{u_1\theta + v_1}{w_1\theta + t_1}=\theta$ and $\det M=1$.
The first equality is equivalent to $(u_1-t_1)\theta+v_1-w_1u=0$ that is $u_1=t_1$ and $v_1=w_1u$. Hence the second equality becomes $u_1^2-uw_1^2=1$, or equivalently 
$u_1+w_1\theta\in\ker N_{\FF_{\ell^2}/\FF_\ell}$ where $N_{\FF_{\ell^2}/\FF_\ell}$ denotes the norm map of the Galois extension $\FF_{\ell^2}/\FF_\ell$. It follows that
\[\left|\mbox{Stab}_{\rm{SL}_2(\FF_\ell)}(\theta)\right|=\left|\ker N_{\FF_{\ell^2}/\FF_\ell}\right|.\]
Now, the norm map is known to be an onto morphism $\FF_{\ell^2}^\times\to\FF_{\ell}^\times$ (see for instance \cite{Jac85} exercise 1, p. 288), then 
\[|\ker N_{\FF_{\ell^2}/\FF_\ell}|=\dfrac{|\FF_{\ell^2}^\times|}{|\FF_{\ell}^\times|}=\dfrac{\ell^2-1}{\ell-1}=\ell+1.
\]
Finally
\[\left|\rm{SL}_2(\FF_\ell). \theta\right|=
\dfrac{|\rm{SL}_2(\FF_\ell)|}{\left|\mbox{Stab}_{\rm{SL}_2(\FF_\ell)}(\theta)\right|}
=\dfrac{(\ell^2-1)(\ell^2-\ell)}{(\ell-1)(\ell+1)}
=\ell^2-\ell
=|\FF_{\ell^2}\setminus\FF_\ell|\]and the proof is complete.
\end{proof}

\begin{corollary}\label{GL}
Assume that $\car\kk=\ell>2$ with $\FF_{\ell^3}\subset\kk$ and $\ell\equiv 3(4)$.
Then $K(\f g_\alpha)$ and $K(\f g_\beta)$ are isomorphic as $\kk$-algebras for all $\alpha,\beta\in\FF_{\ell^3}\setminus\FF_\ell$.
\end{corollary}
\begin{proof}We consider the subgroup $\rm{SL}^{\pm}_2(\FF_\ell)=\{M\in\rm{GL}_2(\FF_\ell)\,;\,\det M=\pm1\}$ and the surjective map $\rm{GL}_2(\Z)\to\rm{SL}^\pm_2(\FF_\ell)$ 
obviously deduced from the canonical map $\rm{SL}_2(\Z)\to\rm{SL}_2(\FF_\ell)$. Applying again 
Proposition \ref{CS} it is sufficient to prove that the induced $\rm{SL}^\pm_2(\FF_\ell)$-action on $\FF_{\ell^3}\setminus\FF_\ell$ is transitive.

Let $\theta\in\FF_{\ell^3}\setminus\FF_{\ell}$. By a dimensional argument $\FF_{\ell^3}=\FF_{\ell}(\theta)$ and $\{1,\theta,\theta^2\}$ is a $\FF_\ell$-basis of $\FF_{\ell^3}$.
An element $M=\left(\begin{smallmatrix}u_1 & v_1 \\ w_1 & t_1 \end{smallmatrix}\right)$ of $\rm{SL}^\pm_2(\FF_\ell)$ lies in $\mbox{Stab}_{\rm{SL}^\pm_2(\FF_\ell)}(\theta)$ if and only if 
$\frac{u_1\theta + v_1}{w_1\theta + t_1}=\theta$ and $\det M=\pm 1$. The first equality is equivalent to $w_1=v_1=0$ and $u_1=t_1$. Hence the second equality becomes $u_1^2=\pm 1$.
Since $-1$ is not a square in $\FF_\ell$ because of the assumption on $\ell$, we deduce $u_1=\pm 1$ and $M=\pm\left(\begin{smallmatrix}1 & 0 \\ 0 & 1 \end{smallmatrix}\right)$.
The identities
\[\left|\mbox{Stab}_{\rm{SL}^\pm_2(\FF_\ell)}(\theta)\right|=2\quad\text{and}\quad|\rm{SL}^\pm_2(\FF_\ell)|=2|\rm{SL}_2(\FF_\ell)|=2\ell(\ell^2-1)\]
finally imply $\left|\rm{SL}_2(\FF_\ell). \theta\right|=\ell^3-\ell=|\FF_{\ell^3}\setminus\FF_\ell|$ and the proof is complete.
\end{proof}

In the case $\ell=2$, the transitivity of the action of $\rm{SL}_2(\FF_\ell)$ can be obtained by direct arguments.
\begin{corollary}Assume that $\car\kk=2$.
\begin{itemize}
\item[(i)] If $\FF_4\subset\kk$, then $K(\f g_\alpha)$ and $K(\f g_\beta)$ are isomorphic as $\kk$-algebras for all $\alpha,\beta\in\FF_4\setminus\FF_2$.
\item[(ii)] If $\FF_8\subset\kk$, then $K(\f g_\alpha)$ and $K(\f g_\beta)$ are isomorphic as $\kk$-algebras for all $\alpha,\beta\in\FF_8\setminus\FF_2$.
\end{itemize}
\end{corollary}
\begin{proof}Let us denote $\FF_4=\{0,1,\omega,\omega^2\}$ where $\omega$ and $\omega^2$ are the roots of the irreducible polynomial $T^2+T+1$ over $\FF_2$. Then $\omega^2=\omega+1=M.\omega$ where $M_1=\left(\begin{smallmatrix}1&1\\0&1\end{smallmatrix}\right)\in\rm{SL}_2(\FF_2)$. Then $\rm{SL}_2(\FF_2)$ acts transitively on $\FF_4\setminus\FF_2$ and the proof is finished as in Corollary \ref{SL}.

We consider now $\FF_8=\{0,1,\omega,\omega^2,\omega^3,\omega^4,\omega^5,\omega^6\}$ for $\omega$ a root of the irreducible polynomial $T^3+T+1$ over $\FF_2$. We have : $\omega^3=\omega+1$, $\omega^4=\omega^2+\omega$, $\omega^5=\omega^2+\omega+1$ and $\omega^6=\omega^2+1$.
Denoting the five non trivial elements of $\rm{SL}_2(\FF_2)$ by
\begin{equation*}M_1=\left(\begin{smallmatrix}1&1\\0&1\end{smallmatrix}\right), \ \ 
M_2=\left(\begin{smallmatrix}1&0\\1&1\end{smallmatrix}\right), \ \ 
M_3=\left(\begin{smallmatrix}0&1\\1&0\end{smallmatrix}\right), \ \ 
M_4=\left(\begin{smallmatrix}0&1\\1&1\end{smallmatrix}\right), \ \ 
M_5=\left(\begin{smallmatrix}1&1\\1&0\end{smallmatrix}\right),\end{equation*}
we compute
$M_1.\omega=\omega^3$, $M_2.\omega=\omega(\omega+1)^{-1}=\omega^5$, $M_3.\omega=\omega^{-1}=\omega^6$, 
$M_4.\omega=(\omega+1)^{-1}=\omega^4$ and $M_5.\omega=(\omega+1)\omega^{-1}=\omega^3\omega^6=\omega^9=\omega^2$
to conclude again that $\rm{SL}_2(\FF_2)$ acts transitively on $\FF_8\setminus\FF_2$.\end{proof}


\section{Valued isomorphisms for enveloping skewfields \\ of the Lie algebras $\f {g}_\alpha$}
\label{continu}

The main goal of classifying Lie algebras $\f g_\alpha$ up to $\kk$-isomorphism of their enveloping skewfields leads naturally to the question 
whether or not the sufficient condition of Proposition \ref{CS} is also necessary. We do not know how to answer it in general. 
In this section we solve in characteristic zero a weaker form of the problem considering isomorphisms of valued skewfields, 
for the valuation canonically associated to the derived subalgebra of the Lie algebras $\f g_\alpha$ as in \ref{DV}.

\subsection{Complete extension of the enveloping skewfield of $\f g_\alpha$}\label{pdo}It is well known and of classical use that the skewfield of fractions 
of an algebra of formal differential operators can be embedded in a skewfield of formal pseudodifferential operators. 
For any field $K$ and any derivation $d$ of $K$, the skewfield  of fractions $Q=K(x\,;\,d)$ of the polynomial algebra of differential operators $A=K[x\,;\,d]$ 
can be embedded in the skewfield $F=K(\!(u\,;\,\delta)\!)$ where $u=x^{-1}$ and $\delta=-d$ (see for instance \cite{Goo}, proposition 5.3). 
We simply recall that the elements of $F$ are Laurent series $\sum_{n>-\infty}a_nu^n$ with coefficients $a_n$ in $K$, the valuation is related to the uniformizer $u$, 
and the commutation law is $ua=au+\sum_{j\geqslant 1}\delta^j(a)u^{j+1}$, which gives rise to $u^{-1}a=au^{-1}-\delta(a)$ or equivalently $xa=ax+d(a)$ for any $a\in K$.\smallskip

For any field $\kk$, we can apply this general process to $U(\f g_\alpha)$ and $K(\f g_\alpha)$ as seen in \eqref{UgOre} and \eqref{UgK} and introduce the skewfield
\begin{equation}\label{UgPDO} F(\f g_\alpha)=\kk(y,z)(\!(u\,;\,\delta_\alpha)\!),\end{equation}
where  
\begin{equation}\label{UgPDObis}u=x^{-1} \quad \text{and} \quad \delta_\alpha=-D_\alpha=-y\partial_y-\alpha z\partial_z.\end{equation}
The valuation $v$ on $K(\f g_\alpha)$ canonically associated to the derived subalgebra of $\f g_\alpha$ and previously defined in \ref{DV} 
is then the discrete valuation associated to $u$ in $ F(\f g_\alpha)$. Therefore:
\begin{equation}v(u)=1\quad\text{ and }\quad v(f)=0 \ \text{ for any } f\in\kk(y,z).
\end{equation}
According to the terminology on valued skewfields, $ F(\f g_\alpha)$ is a local skewfield, that is complete for the topology associated to the discrete valuation~$v$.

\subsection{Valued isomorphisms}Let $\alpha,\beta\in\kk$. Let $v$ and $v'$ the discrete valuations defined as above on $K(\f g_\alpha)$ and $K(\f g_\beta)$ respectively. By definition, we say that $K(\f g_\alpha)$ and $K(\f g_\beta)$ are isomorphic as valued skewfields when there exists an isomorphism of $\kk$-algebras $\varphi:K(\f g_\beta)\to K(\f g_\alpha)$ such that $v(\varphi(s))=v'(s)$ for any $s\in K(\f g_\beta)$ and
$v'(\varphi^{-1}(t))=v(t)$ for any $t\in K(\f g_\alpha)$.
It is clear that such an isomorphism defines by unique extension an isomorphism between $ F(\f g_\beta)$ and $ F(\f g_\alpha)$.
The isomorphisms introduced in Lemma \ref{embedmonomial} are examples of valued isomorphisms, see further Remark \ref{pres}. 

\subsection{Valued rational equivalence for the enveloping algebra of $\f g_\alpha$}
We suppose in this section that $\kk$ is of characteritic zero. We denote by $\overline\kk$ an algebraic closure of $\kk$.
\begin{lemma}\label{puiseux}The notations are those of \ref{notgalpha}. We embed $\kk(y,z)$ in the commutative field $L(y)$, where $L=\bigcup_{n\in\Z_{>0}}\overline\kk(\!(z^{1/n})\!)$ 
is the Puiseux extension of the Laurent series field $\overline\kk(\!(z)\!)$. Let $\alpha$ be an element of $\kk\setminus\Q$. Then:
\begin{itemize}
\item[(i)] There exists a unique $\overline\kk$-derivation of $L(y)$ extending $D_\alpha$, also denoted by $D_\alpha$. It satisfies
$D_\alpha(y^j)=jy^j$ and $D_\alpha(z^{j/n})=\frac jn\alpha z^{j/n}$ for all integers $j\in\Z, n\geqslant 1$.
\item[(ii)] The differential equation $D_\alpha(h)=h$ does not admit non zero solutions in $L$.
\end{itemize}
\end{lemma}
\begin{proof}Point (i) is clear recalling that $D_\alpha$ is defined on $\kk(y,z)$ by $D_\alpha(y)=y$ and $D_\alpha(z)=\alpha z$.
Concerning point (ii) let $h$ be a non zero element of $L$. 
There exist an integer $n\geqslant 1$, an integer $s$ and a sequence  $(\lambda_j)_{j\geqslant s}$ of elements of $\overline\kk$ with $\lambda_s\not=0$ such that
$h=\sum_{j\geqslant s}\lambda_j z^{j/n}$. If $D_\alpha(h)=h$, then $\alpha\frac sn=1$, a contradiction with the assumption $\alpha\notin\Q$.
\end{proof}

\begin{theorem}\label{mainthm}For all $\alpha,\beta\in\kk\setminus\Q$, the following conditions are equivalent.
\begin{itemize}
\item[(i)] $K(\f g_\alpha)$ and $K(\f g_\beta)$ are isomorphic as valued skewfields.
\item[(ii)] $\alpha$ and $\beta$ are in the same orbit for the homographic action of ${\rm GL}_2(\Z)$ on $\kk\setminus\Q$.
\end{itemize}
\end{theorem}
\begin{proof} Lemma \ref{embedmonomial} shows that (ii) implies (i). Conversely suppose that there exists a valued isomorphism of $\kk$-algebras
$\varphi:K(\f g_\beta)\to K(\f g_\alpha)$.  

{\it First step.} As in \ref{plongements} we denote by $x,y,z$ the generators of $K(\f g_\alpha)=\kk(y,z)(x\,;\,D_\alpha)$ and by $x',y',z'$  the generators of $K(\f g_\beta)=\kk(y',z')(x'\,;\,D_\beta)$. The images
\begin{equation}\label{defXYZ}X=\varphi(x'), \ \  Y=\varphi(y'), \ \ Z=\varphi(z')\end{equation}
satisfy the relations $YZ=ZY$, $XY-YX=Y$, $XZ-ZX=\beta Z $ in $K(\f g_\alpha)$. The last two  are equivalent to
\begin{equation}\label{relations}
YX^{-1}-X^{-1}Y=X^{-1}YX^{-1}\quad\text{and}\quad ZX^{-1}-X^{-1}Z=\beta X^{-1}ZX^{-1}.
\end{equation}
By assumption on $\varphi$ we have $v(Y)=v(Z)=0$ and $v(X)=-1$.
Using the embedding of $K(\f g_\alpha)$ in $ F(\f g_\alpha)$ defined in \eqref{UgPDO} and the notations \eqref{UgPDObis} we consider the expansions
\begin{align}
X^{-1}&=c_1 u+c_{2}u^{2}+\cdots,\quad  c_i\in \kk(y,z), \ c_1\not=0,\label{devU}\\
Y&=y_0+y_1u+y_2u^2+\cdots,\quad y_i\in \kk(y,z), \ y_0\not=0,\label{devY}\\
Z&=z_0+z_1u+z_2u^2+\cdots,\quad z_i\in \kk(y,z), \ z_0\not=0.\label{devZ}
\end{align}
Computing in $ F(\f g_\alpha)$ with the usual rules (see for instance section II in \cite{Goo}) and identifying the terms of minimal valuations in the relations \eqref{relations},
we obtain $-c_1\delta_\alpha(y_0)=c_1^2y_0$ and $ -c_1\delta_\alpha(z_0)=\beta c_1^2z_0$.
It follows in particular that we have in $\kk(y,z)$:
\begin{equation}\label{equadiff} 
y_0\notin\kk,\quad z_0\notin\kk\quad\text{and}\quad c_1=\frac1\beta.\frac{D_\alpha(z_0)}{z_0}=\frac{D_\alpha(y_0)}{y_0}.\end{equation}                                                             
We introduce the logarithmic derivation $\Phi_\alpha$ associated to the derivation $D_\alpha$ of $L(y)$. It is defined by 
$\Phi_\alpha(f)=\frac{D_\alpha(f)}{f}$ and satisfies $\Phi_\alpha(fg)=\Phi_\alpha(f)+\Phi_\alpha(g)$ for all non zero elements $f,g$ in $L(y)$.
Hence the differential equation \eqref{equadiff} becomes 
\begin{equation}\label{equadifflog}
\Phi_\alpha(z_0)=\beta\Phi_\alpha(y_0),\quad \text{with }y_0\notin\kk,z_0\notin\kk. 
\end{equation}
{\it Second step.} The element $z_0\in\kk(z)(y)$ is of the form
\begin{equation*}
z_0=a(z)\frac{y^n+g_{n-1}y^{n-1}+\cdots+g_1y+g_0}{y^m+h_{m-1}y^{m-1}+\cdots+h_1y+h_0},\end{equation*}
with $a(z)\in\kk(z)$, $m,n\in\Z_{\geqslant 0}$ and $g_i,h_j\in\kk(z)$. By embedding $\kk(z)$ in the algebraically closed field $L$ defined in Lemma \ref{puiseux} we have in $L(y)$ the factorization
\begin{equation}\label{decz0}
z_0=a(z)\prod_{i=1}^p(y-a_i(z))^{q_i},\end{equation}
where the leading coefficient $a(z)$ is a non zero element of $\kk(z)$, the zeros and poles $a_i(z)$ are pairewise distincts elements in $L$,  
the multiplicities $q_i$ are non zero integers such that $\sum_{i=1}^pq_i=n-m$, and $p$ is a nonnegative integer  (with convention $p=0$ if $z_0\in\kk(z)$). We compute:
\begin{align*}
\Phi_\alpha(z_0)&=\Phi_\alpha(a)+\sum_{i=1}^pq_i\Phi_\alpha(y-a_i)=\Phi_\alpha(a)+\sum_{i=1}^pq_i\frac{D_\alpha(y-a_i)}{y-a_i}\\
&= \Phi_\alpha(a)+\sum_{i=1}^pq_i\frac{y-D_\alpha(a_i)}{y-a_i}=\Phi_\alpha(a)+\sum_{i=1}^pq_i\left(1+\frac{a_i-D_\alpha(a_i)}{y-a_i}\right)\\
&=(q_1 + \cdots+q_p)+\Phi_\alpha(a)+\sum_{i=1}^pq_i\,\frac{a_i-D_\alpha(a_i)}{y-a_i}.\end{align*}Observe that in this sum, the first term is an integer, 
the second one $\Phi_\alpha(a)$ is an element of $\kk(z)$ and the last sum lies in $L(y)$.
Starting similarly from an expression 
\begin{equation}\label{decy0}y_0=b(z)\prod_{j=1}^s(y-b_j(z))^{r_j},\end{equation}
we obtain
$$\Phi_\alpha(y_0)=(r_1 + \cdots+r_s)+\Phi_\alpha(b)+\sum_{j=1}^sr_j\,\frac{b_j-D_\alpha(b_j)}{y-b_j}.$$
We identify the both sides of equation \eqref{equadifflog} considering the expressions  of $\Phi_\alpha(z_0)$ and $\Phi_\alpha(y_0)$ obtained above as their canonical
partial fraction decompositions in $L(y)$.
Assume that there exists  $1\leqslant i\leqslant p$ such that $a_i\not=0$. Then $a_i-D_\alpha(a_i)\not=0$ by assertion (ii) of Lemma \ref{puiseux}. 
It follows that there exists some $1\leqslant j\leqslant s$ such that $b_j\not=0$ and we have in $L(y)$ the equality:
$$ q_i\frac{a_i-D_\alpha(a_i)}{y-a_i}=\beta r_j \frac{b_j-D_\alpha(b_j)}{y-b_j}.$$
This implies $a_i=b_j$ thus $q_i=\beta r_j$ and contradicts the fact that $\beta\notin\Q$. We conclude that all elements $a_i$ and $b_j$ are zero in $L(y)$.
Back to \eqref{decz0} et \eqref{decy0} we have $p=s=1$ and \begin{equation}\label{z0y0}
z_0=a(z)y^q \quad \text{and} \quad y_0=b(z)y^r, \ \text{ with } q,r\in\Z.                                                                           
\end{equation}
where $a$ and $b$ are non zero elements of $\kk(z)$ solutions of the equation\begin{equation}\label{equadiff3}
\displaystyle \Phi_\alpha(a)+q=\beta(\Phi_\alpha(b)+r).\end{equation}
{\it Third step.} The rational functions $a$ and $b$ of $\kk(z)$ can be factorized in $\overline\kk(z)$ as
\begin{equation}\label{ab}a(z)=\lambda\prod_{i=1}^s(z-\lambda_i)^{n_i} \quad\text{and}\quad b(z)=\mu\prod_{j=1}^t(z-\mu_j)^{m_j}\end{equation}
where the leadings coefficients $\lambda$ and $\mu$ are elements of $\kk^\times$, the zeros and poles $\lambda_i$ are pairewise distincts in $\overline\kk$ just like the $\mu_j$'s, 
the exponents $n_i$ and $m_j$ are integers, $s$ and $t$ are nonnegative integers with conventions  $s=0$ if $a\in\kk^\times$ and $t=0$ if $b\in\kk^\times$. Then we compute:
$$\Phi_\alpha(a)=\sum_{i=1}^sn_i\Phi_\alpha(z-\lambda_i)=\sum_{i=1}^sn_i\frac{\alpha z}{z-\lambda_i}=\sum_{i=1}^sn_i\left(\alpha+\frac{\alpha \lambda_i}{z-\lambda_i}\right).$$
Then relation \eqref{equadiff3} implies 
\begin{equation}\label{dces}\alpha\sum_{i=1}^sn_i+q+\sum_{i=1}^s\frac{n_i\alpha \lambda_i}{z-\lambda_i}=
\alpha\beta\sum_{j=1}^tm_j+\beta r+\sum_{j=1}^t\frac{m_j\alpha \beta\mu_j}{z-\mu_j}.\end{equation}
If there exists $1\leqslant i\leqslant s$ such that $\lambda_i\not=0$ then it follows from the unicity of the partial fraction decomposition in $\overline\kk(z)$ 
that there exists $1\leqslant j\leqslant t$ such that  $\lambda_i=\mu_j$ and
$$\frac{n_i\alpha \lambda_i}{z-\lambda_i}=\frac{m_j\alpha \beta\mu_j}{z-\mu_j}=\frac{m_j\alpha \beta\lambda_i}{z-\lambda_i},$$
then $\beta=\frac{n_i}{m_j}\in\Q$ and a contradiction. We deduce that all $\lambda_i$'s and $\mu_j$'s in \eqref{ab} are zero in $\overline\kk$, then  $s=t=1$ and 
$a(z)=\lambda z^{n_1}$ et $b(z)=\mu z^{m_1}$. Denoting simply $n=n_1$ and $m=m_1$ we conclude with \eqref{z0y0} that
\begin{equation}\label{z0y0end}z_0=\lambda z^ny^q \  \text{ and } \ y_0=\mu z^my^r.\end{equation}
{\it Fourth step.} Equality \eqref{dces} reduces to \(\alpha n+q=\beta(\alpha m+r).\) 
If $\alpha m+r=0$, then $m=r=0$ since $\alpha\notin\Q$ and similarly $n=q=0$: this is impossible by \eqref{z0y0end} because $y_0\notin\kk$ and $z_0\notin\kk$  in \eqref{equadifflog}. 
We conclude that $\beta=\frac{n\alpha+q}{m\alpha +r}$.
The property of the matrix with entries $n,q,m,r$ to be invertible follows from the same reasoning applied to the isomorphism $\varphi^{-1}$.\end{proof}

\begin{corollary}\label{corell}We suppose that $\kk=\C$.
The valued isomorphism classes of the skewfields $K(\mathfrak g_\alpha)$ for $\alpha\in\R\setminus\Q$ are in one-to-one correspondence with the sets 
of irrational real numbers having the same continued fraction development after a certain point. 
The valued isomorphism classes of the skewfields $K(\mathfrak g_\alpha)$ for $\alpha\in\C\setminus\R$ are in one-to-one correspondence with the isomorphism classes of complex elliptic curves.
\end{corollary}
\begin{proof}Follows directly from Theorem \ref{mainthm} and Remark \ref{ell}.
\end{proof}

\begin{remark}\label{pres}The valued isomorphisms of Lemma \ref{embedmonomial} satisfy the particular property
of stabilizing the subfield $\kk(y,z)$. However composing such an isomorphism by an inner automorphism of $K(\f g_\alpha)$ gives rise to a valued isomorphism which does not necessarily stabilize $\kk(y,z)$.
\end{remark}

\begin{remark}
It can be proved (by arguments similar to those used in \cite{FDFM}, Proposition 1.1.6, or \cite{AD95}, Theorem 2.3) that any isomorphism between two skewfields of formal 
pseudodifferential operators necessarily preserves the valuations. Without detailing here the proof of this general result, we observe that it implies that the equivalent 
conditions (i) and (ii) of Theorem \ref{mainthm} are also equivalent to the property of  $F(\f g_\alpha)$ and $F(\f g_\beta)$ to be $\kk$-isomorphic.\end{remark}


\section{Enveloping skewfield of the Lie algebra $\f  q$}
\label{qq}

\subsection{Notations}\label{notq}In this section,  $\kk$ is an arbitrary field. We consider the Lie algebra $\f q$ over $\kk$ whose brackets on a basis $\{x,y,z\}$ are 
defined by relations \eqref{brakq}. The enveloping algebra $U(\f q)$ is the associative $\kk$-algebra generated by three generators  $x,y,z$ and relations
\begin{equation}\label{Uqrel}yz=zy,\quad xy-yx=y,\quad xz-zx=y+ z.\end{equation}
It can be viewed as the iterated Ore extension
\begin{equation}\label{UqOre}U(\f q)=\kk[y,z][x\,;\,\Delta],\quad \text{where } \Delta=y\partial_y+(y+ z)\partial_z.\end{equation}
Denoting again by $\Delta$ the extension of $\Delta$ to a $\kk$-derivation of the field $\kk(y,z)$, the skewfield of fractions of $U(\f q)$ is 
\begin{equation}\label{UqK}K(\f q)=\kk(y,z)(x\,;\,\Delta),
\end{equation}
and as in \ref{pdo}, it can be embedded in the local skewfield
\begin{equation}\label{UqF} F(\f q)=\kk(y,z)(\!(u\,;\,-\Delta)\!) \ \ \text{with } u=x^{-1}.
\end{equation}
In order to simplify the commutation relations \eqref{Uqrel}, we set
\begin{equation}t=y^{-1}z,\end{equation}
which satisfies
\begin{equation}\label{kyt}\kk(y,z)=\kk(y,t)\quad\text{and}\quad \Delta=y\partial_y+\partial_t.\end{equation}
Hence \eqref{UqK} and \eqref{UqF} become
\begin{equation}\label{UqKF}K(\f q)=\kk(y,t)(x\,;\,\Delta) \ \subset \  F(\f q)=\kk(y,t)(\!(u\,;\,-\Delta)\!)\end{equation}
 with $u=x^{-1}$ and commutation relations  
\begin{equation}\label{Vqrel}yt=ty,\quad xy-yx=y,\quad xt-tx=1.\end{equation}

\subsection{Center and Gelfand-Kirillov property}
The study splits into two cases depending on the characteristic.
We start with the following lemma.
\begin{lemma}\label{theker}The kernel of the derivation $\Delta$ of $\kk(y,z)$ is equal to $\kk$ when~$\kk$ is of characteristic zero, and to $\kk(y^\ell,z^\ell)$ when $\kk$ is of characteristic $\ell>0$.
\end{lemma}
\begin{proof}Let $f$ be a non zero element of $\kk(y,z)$ such that $\Delta(f)=0$. By \eqref{kyt} we can consider the expansion of $f$ in $\kk(t)(\!(y)\!)$ as $f=\sum_{j\geqslant j_0}a_jy^j$ with $a_j\in\kk(t)$ for all integers $j\geqslant j_0$.
Applying $\Delta$ we deduce that $\sum_{j\geqslant j_0}(\Delta(a_j)+ja_j)y^j=0$. 
Since the restriction of $\Delta$ to $\kk(t)$ is the ordinary derivative $\partial_t$, this leads for any $j$ in the support of $f$ to the differential equation 
\begin{equation}\label{equadiff5}\partial_t(a_j)=-ja_j\quad\text{with} \ a_j\in\kk(t), a_j\not=0.
\end{equation}
Denoting by $\overline{\kk}$ an algebraic closure of $\kk$, the rational function $a_j$ can be factorized in $\overline\kk(t)$ as $a_j(t)=\lambda\prod_{i=1}^s(t-\lambda_i)^{n_i}$ 
where $\lambda\in\kk^\times$, $s\geqslant 0$ (with convention  $s=0$ if $a_j\in\kk^\times$), $n_i\in\Z$, the zeros and poles $\lambda_i$ pairewise distincts in $\overline\kk$.
Applying the logarithmic derivation $\Psi$ associated to the canonical extension of $\partial_t$ to $\overline\kk(t)$, we compute
$$\Psi(a_j)=\sum_{i=1}^sn_i\Psi (t-\lambda_i)=\sum_{i=1}^sn_i\frac{1}{t-\lambda_i}.$$
By  \eqref{equadiff5}, we deduce that $j$ and all $n_i$ vanish in $\kk$. When $\kk$ is of characteristic zero, it follows that $f\in\kk$. When $\kk$ is of characteristic $\ell>0$, 
we obtain $f\in\kk(t^\ell)(\!(y^\ell)\!)$ and conclude using assertion (ii) of Lemma \ref{lemmecomm} that $f\in\kk(t^\ell,y^\ell)$. Then the proof is complete recalling \eqref{kyt}.\end{proof}
In the case of the zero characteristic this result is proved in \cite[Proposition 1.2.4.1]{LR}) by a different method based on
\cite[Theorem 2.1]{Now}.

\begin{proposition}\label{GKqzero}If $\kk$ is of characteristic zero, then the center $C(\f q)$ of $K(\f q)$ is equal to $\kk$, and the Lie algebra $\f q$ does not satisfy the Gelfand-Kirillov property.
\end{proposition}
\begin{proof} By \eqref{UqOre} and \eqref{UqK} it follows from  \cite[Theorem 5.6]{Goo} that $C(\f  q)$ is the kernel of the derivation $\Delta$ of $\kk(y,z)$. 
We conclude with Lemma \ref{theker} that $C(\f  q)=\kk$, and then $K(\f q)$ cannot be isomorphic to a Weyl skewfield $\mathcal D_{1,1}(\kk)$  by a dimensional argument already used in the proof of Corollary \ref{isoaq}.\end{proof}

\begin{theorem}\label{GKql}We suppose that $\kk$ is of characteristic $\ell >0$.
\begin{itemize}
 \item[(i)] \label{zellq}
The center $C(\f q)$ of $K(\f q)$ is equal to $\kk(y^\ell,z^\ell,(x^\ell -x)^\ell)$.
In particular, $K(\f q)$ is of dimension $\ell^4$ over its center.
 \item[(ii)]\label{Daq}
The Lie algebra $\f q$ does not satisfy the Gelfand-Kirillov property.
\end{itemize}
\end{theorem}

\begin{proof} We start computing  the centralizer $\mathcal C(x)$ of $x$ in $K(\f q)=\kk(y,z)(x\,;\,\Delta)$.
Applying \cite[Theorem 5.8]{Goo} we deduce that $\mathcal C(x)=F(x)$ for $F$ the kernel of $\Delta$. Then Lemma \ref{theker} implies that $\mathcal C(x)$ is the commutative subfield $\kk(x,y^\ell,z^\ell)$.
Moreover routine inductions based on relations \eqref{Uqrel} show that 
\begin{equation}\label{comreli}xy^i=y^i(x+i),\quad x^iy=y(x+1)^i,\quad x^it-tx^i=ix^{i-1}\end{equation}
for any $i\geqslant 1$ and then
\begin{equation}\label{comrell}(x^\ell-x)y=y(x^\ell-x)\quad\text{and}\quad(x^\ell-x)t=t(x^\ell-x)-1.\end{equation}
Hence $(x^\ell-x)^\ell$ is central in $K(\f q)$ and denoting $C_\ell=\kk(y^\ell,z^\ell,(x^\ell-x)^\ell)$ we have the following inclusions:
\begin{equation}C_\ell\subseteq C(\f q)\subseteq \mathcal C(x)\subseteq K(\f q).\end{equation}
It follows that the dimension $d$ of $K(\f q)$ over its center $C(\f q)$ divides $\ell^4$, thus can be equal to $1$, $\ell^2$ or $\ell^4$. 
The case $d=1$ is obviously excluded because $K(\f q)$ is non commutative. Suppose that $d=\ell^2$. Then
$$[K(\f q):\mathcal C(x)][\mathcal C(x):C(\f q)]=\ell^2.$$
Since $\mathcal C(x)=\kk(x,y^\ell,z^\ell)$, we have $[K(\f q):\mathcal C(x)]=\ell^2$. Thus $[\mathcal C(x):C(\f q)]=1$ which is impossible because $x\notin C(\f q)$. 
We conclude that $d=\ell^4$. In other words $C(\f q)=C_\ell$ and point (i) is proved.
Assertion (ii) follows by the same dimensional argument as in Theorem \ref{GKlbis}.\end{proof}

\begin{corollary}\label{sepqga}For any commutative field $\kk$ and any element $\alpha$ of the prime subfield $\kk_0$ of $\kk$, the skewfields $K(\f g_\alpha)$ and $K(\f q)$ are not isomorphic.\end{corollary}
\begin{proof}Follows directly from Corollary \ref{isoaq}, Proposition \ref{GKqzero} and Theorems \ref{GKl} and \ref{GKql}.
\end{proof}

As in Section \ref{CGKell} for $K(\f g_\alpha)$, the following proposition describes the structure of $K(\f q)$ over its center.
\begin{proposition}
\label{strucq}
We suppose that $\kk$ is of characteristic $\ell >0$.
Let $L$ be the skewfield generated by $t$, $y^\ell$ and $x^\ell-x$ in $K(\f q)$.
\begin{itemize}
 \item[(i)] The center of $L$ is $C(\f q)$.
 \item[(ii)] The centralizer of $L$ in $K(\f q)$ is the skewfield $L'$ generated by $y$, $t^\ell$ and $x^\ell$.
 \item[(iii)] $L$ and $L'$ are both isomorphic to a Weyl skewfield $\mathcal D_{1,1}(\kk)$.
 \item[(iv)] $K(\f  q)$ is isomorphic to the tensor product of $L$ and $L'$ over $C(\f q)$.
\item[(v)] The class $[K(\f q)]$ of $K(\f q)$ in the Brauer group of $C(\f q)$ is of order $\ell$.
\end{itemize}
\end{proposition}
\begin{proof}
Similar to that of Proposition \ref{struc1} and Corollary \ref{brauer} using relations \eqref{comreli} and \eqref{comrell}.
\end{proof}

\subsection{Separation of  $K(\f q)$ and $K(\f g_\alpha)$ as valued skewfields}By Corollary \ref{sepqga}, the isomorphism problem of the skewfields
$K(\f q)$ and $K(\f g_\alpha)$ remains open only when $\alpha\notin\kk_0$. As in Section \ref{continu} above, we solve it in the weaker following form.
\begin{theorem}We suppose that $\kk$ is of characteristic zero. For any $\alpha\in\kk$ such that $\alpha\notin\Q$, $K(\f q)$ and $K(\f g_\alpha)$ are not isomorphic as valued skewfields.
\end{theorem}

\begin{proof} We proceed by contradiction supposing that there exists a valued isomorphism $\varphi:K(\f  g_\beta)\to K(\f  q)$. 
We adapt {\it mutatis mutandis} calculations of the proof of Theorem \ref{mainthm} by considering the images 
\begin{equation}X=\varphi(x'), \ \  Y=\varphi(y'), \ \ Z=\varphi(z')\end{equation}
of the generators $x',y',z'$ of $K(\f  g_\beta)$. They satisfy in $K(\f  q)$ the relations
 \begin{equation}\label{relQq}
YZ=ZY,\ \ XY-YX=Y,\ \ XZ-ZX=\beta Z.\end{equation}
We use here the description \eqref{UqKF} of $K(\f  q)$ and $F(\f  q)$. The expansions  of $X,Y,Z$ in $F(\f  q)$ are of the same form as in relations
\eqref{devU}, \eqref{devY} and \eqref{devZ} with coefficients $b_i,y_i,z_i$ in $\kk(y,t)$. By identification of both sides in the commutation relations, 
we obtain similarly  to \eqref{equadiff} the differential equation in $\kk(y,t)$ :\begin{equation}\label{equadiff2}
y_0\notin\ker\Delta, \quad z_0\notin\ker\Delta, \quad\text{and}\quad \frac{\Delta(z_0)}{z_0}=\beta\,\frac{\Delta(y_0)}{y_0}.                                                             
\end{equation}
We introduce the logarithmic derivation $\Psi$ defined by $\Psi(f)=\frac{\Delta(f)}f$ for any $f\in\kk(y,t)$ to rewrite it as
\begin{equation}\label{equadifflog2}
\Psi(z_0)=\beta\Psi(y_0). 
\end{equation}
We embed $\kk(t)$ in the algebraically closed field  $L=\bigcup_{n\in\Z_{> 0}}\overline\kk(\!(t^{1/n})\!)$ and introduce the factorizations in $L(y)$ of the elements $y_0$ and $z_0$ of $\kk(t)(y)$
$$z_0=a(t)\prod_{i=1}^p(y-a_i(t))^{q_i}\qquad\text{et}\qquad y_0=b(t)\prod_{j=1}^s(y-b_j(t))^{r_j} ,$$
where $a,b$ non zero in $\kk(t)$, the zeros and poles $a_i$ are pairewise distincts elements in $L$ as the $b_j$'s,
with the same conventions on the notations as in formulae \eqref{decz0} and \eqref{decy0}. Since $\Delta(y)=y$, calculations similar to those in the second step of the proof of 
Theorem \ref{mainthm} show from \eqref{equadiff2} that there exist integers $q,r\in\Z$ such that
\begin{equation}\label{y0z0}
z_0=a(t)y^{q}\quad\text{and}\quad y_0=b(t)y^{r},\qquad a,b\in \kk(t).
\end{equation}
Thus $a$ and $b$ are non zero elements of $\kk(t)$ solutions of the differential equation
\begin{equation}\label{equadiff4}
\Psi(a)+q=\beta(\Psi(b)+r).\end{equation}
We factorize $a$ and $b$ in $\overline{\kk}(t)$ in the form
$$a(t)=\lambda\prod_{i=1}^s(t-\lambda_i)^{n_i} \quad\text{and}\quad b(t)=\mu\prod_{j=1}^{s'}(t-\mu_j)^{m_j}$$
where $\lambda,\mu\in\kk^\times$, the zeros and poles $\lambda_i\in\overline{\kk}$ are paisewise distincts as the $\mu_j$'s , 
the exponents $n_i$ and $m_j$ are non zero in $\Z$ with conventions $s=0$ if $a\in\kk^\times$ and $s'=0$ if $b\in\kk^\times$. We have
$$\Psi(a)=\sum_{i=1}^sn_i\Psi(t-\lambda_i)=\sum_{i=1}^sn_i\frac1{t-\lambda_i}.$$
Identity \eqref{equadiff4} becomes 
\begin{equation}\label{dces2}q+\sum_{i=1}^s\frac{n_i}{t-\lambda_i}=\beta r+\sum_{j=1}^{s'}\frac{\beta m_j}{t-\mu_j}.\end{equation}
This implies $q=\beta r$ then $q=r=0$ since $\beta\notin\Q$.
If there exists  $1\leqslant i\leqslant s$ such that $n_i\neq0$, it
follows from the unicity of the partial fraction decomposition in $\overline\kk(t)$ 
that there exists $1\leqslant j\leqslant s'$ such that $m_j\neq0$ and $\mu_j=\lambda_i$. We obtain by identification $$\frac{n_i}{t-\lambda_i}=\frac{ \beta m_j}{t-\mu_j}=\frac{ \beta m_j}{t-\lambda_i}$$ 
then $\beta = n_i (m_j)^{-1}$ and a contradiction with the assumption $\beta\notin\Q$. We deduce that all integers $n_i$ and $m_j$ are zero, hence $a=\lambda$ et $b=\mu$. 
Finally the equalities in \eqref{y0z0} reduce to $z_0=\lambda$ and $y_0=\mu$. 
This is impossible since $y_0$ and $z_0$ are not in $\ker\Delta$ by \eqref{equadiff2}. 
\end{proof}


\section*{Additional comment}

As seen in \ref{prelimGK}, the classification of enveloping skewfields of Lie algebras of dimension 3 over an algebraically closed field $\kk$ is limited to the Weyl skewfield $\mathcal D_{1,1}(\kk)$ 
for the three classical examples $\f h$, $\f b$ and $\f{sl}(2)$, and to the skewfields $K(\f  g_\alpha)$ and $K(\f q)$ studied above. 
The results presented in sections 2, 3 and 4 are however all proved without the assumption that $\kk$ is algebraically closed. 
When $\kk$ is no longer algebraically closed, the classification of \ref{Jac} according to the dimension $d$ of the derived Lie subalgebra reveals two other families of Lie algebras. 
The case $d=3$  has been studied extensively in the article \cite{Mal}. 
In the case $d=2$ appears a new family of Lie algebras indexed when $\car\kk\not=2$ by a couple of scalars $(p,q)\in\kk\times \kk^\times$ such that $p^2-4q$ is not a square in $\kk$ and 
whose Lie brackets on  a basis $\{x,y,z\}$ are given by
\begin{equation*}\label{case3}
[x,y]=-qz,\quad [x,z]=y+pz,\quad [y,z]=0.
\end{equation*}
Exploratory results show that the situation splits into two cases depending on whether the parameter $p$ is zero or not, leading to potential further study.
\medskip


\bibliographystyle{plain}
\bibliography{ADL_13septembre}

\end{document}